\pdfoutput=1
\RequirePackage{ifpdf}
\ifpdf 
\documentclass[pdftex]{sigma}
\else
\documentclass{sigma}
\fi

\usepackage[all]{xy}

\numberwithin{equation}{section}

\newtheorem{Theorem}{Theorem}[section]
\newtheorem{Lemma}[Theorem]{Lemma}
\newtheorem{Proposition}[Theorem]{Proposition}
\theoremstyle{definition}
\newtheorem{Definition}[Theorem]{Definition}

\newtheorem{Remark}[Theorem]{Remark}

\begin{document}
\allowdisplaybreaks

\newcommand{\arXivNumber}{1901.01566}

\renewcommand{\thefootnote}{}

\renewcommand{\PaperNumber}{034}

\FirstPageHeading

\ShortArticleName{Jacobian Conjecture via Differential Galois Theory}

\ArticleName{Jacobian Conjecture via Differential Galois Theory\footnote{This paper is a~contribution to the Special Issue on Algebraic Methods in Dynamical Systems. The full collection is available at \href{https://www.emis.de/journals/SIGMA/AMDS2018.html}{https://www.emis.de/journals/SIGMA/AMDS2018.html}}}

\Author{El\.zbieta ADAMUS~$^\dag$, Teresa CRESPO~$^\ddag$ and Zbigniew HAJTO~$^\S$}

\AuthorNameForHeading{E.~Adamus, T.~Crespo and Z.~Hajto}

\Address{$^\dag$~Faculty of Applied Mathematics, AGH University of Science and Technology,\\
\hphantom{$^\dag$}~al.~Mickiewicza 30, 30-059 Krak\'ow, Poland}
\EmailD{\href{mailto:esowa@agh.edu.pl}{esowa@agh.edu.pl}}

\Address{$^\ddag$~Departament de Matem\`atiques i Inform\`atica, Universitat de Barcelona,\\
\hphantom{$^\ddag$}~Gran Via de les Corts Catalanes 585, 08007 Barcelona, Spain}
\EmailD{\href{mailto:teresa.crespo@ub.edu}{teresa.crespo@ub.edu}}

\Address{$^\S$~Faculty of Mathematics and Computer Science, Jagiellonian University,\\
\hphantom{$^\S$}~ul.~{\L}ojasiewicza 6, 30-348 Krak\'ow, Poland}
\EmailD{\href{mailto:zbigniew.hajto@uj.edu.pl}{zbigniew.hajto@uj.edu.pl}}

\ArticleDates{Received January 23, 2019, in final form May 01, 2019; Published online May 03, 2019}

\Abstract{We prove that a polynomial map is invertible if and only if some associated differential ring homomorphism is bijective. To this end, we use a theorem of Crespo and Hajto linking the invertibility of polynomial maps with Picard--Vessiot extensions of partial differential fields, the theory of strongly normal extensions as presented by Kovacic and the characterization of Picard--Vessiot extensions in terms of tensor products given by Levelt.}


\Classification{14R10; 14R15; 13N15; 12F10}

\renewcommand{\thefootnote}{\arabic{footnote}}
\setcounter{footnote}{0}

\section{Introduction}

The Jacobian conjecture originates from the problem posed by Keller in~\cite{Ke}. It is the 16th problem in Stephen Smale's list of mathematical problems for the twenty-first century (cf.~\cite{Sma}). Let us recall the precise statement of the Jacobian conjecture.

Let $C$ denote an algebraically closed field of characteristic zero. Let $n>0$ be a fixed integer and let $F=(F_1, \ldots, F_n)\colon C^n \rightarrow C^n$ be a polynomial map, i.e., $F_i \in C[X]$, for $i=1,\ldots, n$, where $X=(X_1, \ldots, X_n)$. We consider the Jacobian matrix $J_F = \big[\frac{\partial F_i}{\partial X_j}\big]_{1 \leq i,j \leq n}$. The Jacobian conjecture states that if $\det(J_F)$ is a non-zero constant, then $F$ has an inverse map, which is also polynomial.

In spite of many different approaches involving various mathematical tools the question is still open. In 1980 S.S.-S.~Wang \cite{W} proved the Jacobian conjecture for quadratic maps. The same year H.~Bass, E.~Connell and D.~Wright in~\cite{Bass} and A.V.~Yagzhev in~\cite{Y} independently reduced the Jacobian problem to maps of degree 3 at cost of enlarging the number of variables. In~\cite{Bass} an interesting differential approach to the Jacobian problem due to Wright is also presented. An account of the research on the Jacobian conjecture may be found in \cite{E} and in the survey~\cite{E2}. In the recent years some new achievements have been reached such as the negative answer to the long standing dependence problem given by M.~de Bondt~\cite{Bo} and results by several authors on classification of special types of Keller maps. Recently, in~\cite{ABCH} and~\cite{ABCH2} we have considered the class of Pascal finite automorphisms. On the other hand the conjecture holds under strong additional assumptions. As an example let us recall the result of L.A.~Campbell (see~\cite{C}), which states that the thesis holds if $C(X)/C(F)$ is a Galois extension.

In a previous paper using Picard--Vessiot theory of partial differential fields Crespo and Hajto obtained a differential version of the classical theorem of Campbell. Let us consider the field~$C(X)$ with the differential structure given by the Nambu derivations (see Section~\ref{section3}). Then Crespo and Hajto proved that if the differential extension~$C(X)/C$ is Picard--Vessiot, then~$F$ is invertible (cf.~\cite[Theorem~2]{CH}). A computational approach to this result using wronskians has been given in~\cite{ABH}. The use of wronskians makes the calculations longer but allows application to the more general context of dominant polynomial maps without the assumption of the Jacobian determinant beeing a non-zero constant to check the Galois character of the associated field extension.

In \cite[Theorem~1]{L}, Levelt proved a necessary condition for a differential field extension $K/k$ to be Picard--Vessiot in terms of tensor products.

In this paper, we prove a partial converse of Levelt's theorem. To this end, we use the theory of strongly normal extensions as presented by Kovacic in~\cite{Kov2} and~\cite{Kov1}. By using the converse of Levelt's theorem together with the above mentioned result by Crespo and Hajto we obtain that a polynomial map is invertible if and only if some associated differential ring homomorphism is bijective. This provides a criterion to check the invertibility of polynomial maps.

\section{Inverting a theorem of A.H.M.~Levelt}\label{section2}

In this section we prove that Levelt's necessary condition for a differential field extension $K/k$ to be Picard--Vessiot is sufficient for $K/k$ to be strongly normal in the case in which the base field is a field of constants. In the next section we shall apply this result to the extension $C(X)/C$ endowed with the Nambu derivations associated to a polynomial map $F$ in order to obtain an equivalent condition to the invertibility of~$F$.

Let $C$ be an algebraically closed field of characteristic zero. Let $R$ be an integral domain containing $C$ with the differential ring structure given by a finite set $\Delta$ of commuting derivations. Let $K$ be the field of fractions of $R$ with the differential structure given by extending the derivations in $\Delta$ in the standard way. Let us assume that $C$ is the field of constants of $K$ and that $K$ is differentially finitely generated over $C$. Any derivation $\delta \in \Delta$ extends to $K\otimes_C R$ by the formula $\delta (x\otimes y)=\delta(x) \otimes y + x \otimes \delta(y)$ on elementary tensors and by linearity to the whole tensor product. Let $E$ denote the differential subring of constants in $K \otimes_C R$. We consider the differential ring homomorphism
\begin{gather*} \phi \colon \ K \otimes_C E \rightarrow K \otimes_C R, \qquad \phi (a \otimes d)=(a \otimes 1)d .\end{gather*}

Our aim is to prove that if $\phi$ is an isomorphism, then the extension $K/C$ is strongly normal. We shall prove first that under the above hypothesis, $\phi$ is injective. To this end, we need the following lemma.

\begin{Lemma} \label{ideal} The map
 \begin{gather*} h\colon \ E \rightarrow K \otimes_C E, \qquad h(d ) = 1 \otimes d,\end{gather*}
 induces a bijection between the set $\mathcal{I}(E)$ of ideals of $E$ and the set $\mathcal{I}(K \otimes_C E)$ of differential ideals of $K \otimes_C E$.
\end{Lemma}

\begin{proof}To an ideal $\mathfrak{a}$ of $E$ we associate its extension $\mathfrak{a}^e=K \otimes_C \mathfrak{a}$ and to an ideal $\mathfrak{b}$ of $K \otimes_C E$ its contraction $\mathfrak{b}^c=h^{-1}(\mathfrak{b})$.

The inclusion $\mathfrak{a}^{ec} \supset \mathfrak{a}$ is well known. Let us prove $\mathfrak{a}^{ec} \subset \mathfrak{a}$. We take a basis $\Lambda$ of the $C$-vector space $\mathfrak{a}$ and extend it to a basis $M$ of the $C$-vector space $E$. Then $1 \otimes M$ is a basis of the $K$-vector space $K \otimes_C E$. Let $d$ be any element in $\mathfrak{a}^{ec}$. Then $1 \otimes d \in \mathfrak{a}^{ece}=\mathfrak{a}^{e}$, so
\begin{gather*} 1 \otimes d = \sum_{\lambda \in \Lambda} r_{\lambda} \otimes \lambda, \qquad r _{\lambda} \in K.\end{gather*}
On the other hand $d=\sum\limits_{\mu \in M} c_{\mu} \mu$, $c_{\mu} \in C$, so
\begin{gather*}1 \otimes d=\sum_{\mu \in M} 1 \otimes c_{\mu} \mu=\sum_{\mu \in M} c_{\mu} \otimes \mu.\end{gather*}
Comparing the coefficients in both expressions, we get that $c_{\mu}=0 $ for $\mu \in M {\setminus} \Lambda$ and $r_{\lambda}=c_{\lambda}$, for $\lambda \in \Lambda$.
Hence $d \in \mathfrak{a}$.

The inclusion $\mathfrak{b}^{ce} \subset \mathfrak{b}$ is well known. Suppose now $\mathfrak{b} {\setminus} \mathfrak{b}^{ce} \neq \varnothing$. We take a $C$-vector space basis~$\Lambda$ of $\mathfrak{b}^c$ and extend it to a basis $M$ of the $C$-vector space~$E$. Let us choose an element $a \in \mathfrak{b} {\setminus} \mathfrak{b}^{ce}$ such that its representation in the form
\begin{gather*} a = \sum_{\mu \in M} r_{\mu} \otimes \mu, \qquad r_{\mu} \in K\end{gather*}
has the smallest number of nonzero terms.

First let us consider the case when $a$ is an elementary tensor, i.e., $a=r \otimes \mu$, for $a \in K$, $\mu \in M$. If $\mu \in \Lambda$, then $a \in \mathfrak{b}^{ce}$ and we have a contradiction. So let us assume that $\mu \in M {\setminus} \Lambda$. Then we multiply $a$ by $r^{-1} \in K$ and obtain that $1 \otimes \mu \in \mathfrak{b}$ and consequently $\mu \in \mathfrak{b}^c$, hence again $a \in \mathfrak{b}^{ce}$ and we have a contradiction.

Let us assume now that the representation $a = \sum\limits_{\mu \in M} r_{\mu} \otimes \mu$, $r_{\mu} \in K,$ has at least two nonzero terms. Since $\mathfrak{b}$ is a differential ideal, then for any differential operator $\delta \in \Delta$ we have
\begin{gather*} \delta a = \sum_{\mu \in M} \delta r_{\mu} \otimes \mu \in \mathfrak{b}.\end{gather*}
Since $a \neq 0$, we can choose $\mu_0$ such that $r_{\mu_0} \neq 0$. Then $\delta r_{\mu_0}a-r_{\mu_0} \delta a \in \mathfrak{b}$ and
\begin{align*}\delta r_{\mu_0}a-r_{\mu_0} \delta a &= \delta r_{\mu_0} \bigg( \sum_{\mu \in M} r_{\mu} \otimes \mu\bigg)- r_{\mu_0} \bigg( \sum_{\mu \in M} \delta r_{\mu} \otimes \mu \bigg)\\
&= \sum_{\mu \in M,\, \mu \neq \mu_0} (\delta r_{\mu_0} r_{\mu} - r_{\mu_0} \delta r_{\mu} ) \otimes \mu.\end{align*}
since the coefficient of $\mu_0$ is equal to $(\delta r_{\mu_0})r_{\mu_0}-r_{\mu_0} (\delta r_{\mu_0})=0$. By the minimality assumption on $a \in \mathfrak{b} {\setminus} \mathfrak{b}^{ce} $, we have $\delta r_{\mu_0}a-r_{\mu_0} \delta a \in \mathfrak{b}^{ce}$, hence

\begin{gather*}\delta r_{\mu_0}a-r_{\mu_0}\delta a=\sum_{\mu \in \Lambda, \,\mu \neq \mu_0} (\delta r_{\mu_0} r_{\mu} - r_{\mu_0} \delta r_{\mu} ) \otimes \mu.\end{gather*}
We have then $\delta r_{\mu_0} r_{\mu} - r_{\mu_0} \delta r_{\mu}=0$ for $\mu \in M {\setminus} \Lambda$. Hence $\delta\big(\frac{r_{\mu}}{r_{\mu_0}}\big)=0$ for $\mu \in M {\setminus} \Lambda$. This means $\frac{r_{\mu}}{r_{\mu_0}}$ is a constant in $K$. Hence there exists $c_{\mu} \in C$ such that $r_{\mu}=c_{\mu} r_{\mu_0}$ for $\mu \in M {\setminus} \Lambda$. We obtain
\begin{gather*}a=\sum_{\mu \in M {\setminus} \Lambda} c_{\mu} r_{\mu_0} \otimes \mu+\sum_{\mu \in \Lambda} r_{\mu} \otimes \mu.\end{gather*}
Observe that $\sum\limits_{\mu \in \Lambda} r_{\mu} \otimes \mu \in \mathfrak{b}^{ce} \subset \mathfrak{b}$. Hence
\begin{gather*} \mathfrak{b} \ni \sum_{\mu \in M {\setminus} \Lambda} c_{\mu} r_{\mu_0} \otimes \mu=(r_{\mu_0} \otimes 1)
\bigg(1 \otimes \sum_{\mu \in M {\setminus} \Lambda} c_{\mu} \mu\bigg).\end{gather*}
Since $K$ is a field, we get that $1 \otimes d_0 \in \mathfrak{b}$, for $d_0=\sum\limits_{\mu \in M {\setminus} \Lambda} c_{\mu} \mu$. So $d_0 \in \mathfrak{b}^c$ and $(r_{\mu_0} \otimes 1)(1 \otimes d_0) \in \mathfrak{b}^{ce} $. And we have a contradiction with the minimality assumption on $a \in \mathfrak{b} {\setminus} \mathfrak{b}^{ce} $.
\end{proof}

\begin{Proposition}\label{injective}The map
\begin{gather*} \phi \colon \ K \otimes_C E \rightarrow K \otimes_C R, \qquad \phi (a \otimes d)=(a \otimes 1)d\end{gather*}
is injective.
\end{Proposition}

\begin{proof} Denote $\mathfrak{b}= \ker \phi$. Using Lemma~\ref{ideal} we can assume that $\mathfrak{b} = K \otimes \mathfrak{c}$, where $\mathfrak{c} \in \mathcal{I}(E)$. Take $c \in \mathfrak{c}$. Then
\begin{gather*}0=\phi(1 \otimes c) = (1 \otimes 1)c=c.\end{gather*} So $c=0$ and $\mathfrak{b}=(0)$.
\end{proof}

We recall the notion of almost constant differential ring which will be used in the sequel.

\begin{Definition}[{\cite[Definition~5.1]{Kov2}}] Let $A$ be a differential ring and~$C_A$ its ring of constants. We say that $A$ is almost constant if the inclusion $C_A\subset A$ induces a bijection between the set of radical ideals of $C_A$ and the set of radical differential ideals of~$A$.
\end{Definition}

\begin{Proposition} Let $C$ be an algebraically closed field of characteristic zero. Let $R$ be an integral differential ring containing $C$ and let $K$ be the field of fractions of~$R$. We assume that~$C$ is the field of constants of $K$ and that $K$ is differentially finitely generated over $C$. If the differential morphism
\begin{gather*} \phi \colon \ K \otimes_C E \rightarrow K \otimes_C R, \qquad \phi (a \otimes d)=(a \otimes 1)d\end{gather*}
is an isomorphism, then the differential ring $K\otimes_C R$ is almost constant.
\end{Proposition}

\begin{proof} If $\phi$ is a differential isomorphism, there is a bijection between the set of radical differential ideals of $K\otimes_C R$ and the set of radical differential ideals of $K \otimes_C E$. By Lemma~\ref{ideal} and \cite[Proposition~3.4]{Kov2}, this last set is in bijection with the set of radical ideals of the ring of constants~$E$ of $K\otimes_C R$.
\end{proof}

\begin{Theorem} \label{mainthm} Let $C$ be an algebraically closed field of characteristic zero. Let $R$ be an integral differential ring containing $C$ and let $K$ be the field of fractions of~$R$. We assume that $C$ is the field of constants of $K$ and that $K$ is differentially finitely generated over~$C$. If the differential morphism
\begin{gather*} \phi \colon \ K \otimes_C E \rightarrow K \otimes_C R, \qquad \phi (a \otimes d)=(a \otimes 1)d\end{gather*}
is an isomorphism, then $K/C$ is a strongly normal extension.
\end{Theorem}

\begin{proof}To prove that $K/C$ is strongly normal, we shall apply \cite[Proposition 12.5]{Kov2}. Let $\sigma$ be an arbitrary $\Delta$-isomorphism of $K$ over $C$. We put $\sigma\colon K \rightarrow M$, where $M$ is any differential field extension of $K$ and denote by $D_{\sigma}$ the field of constants of $M$. Define $\overline{\sigma}\colon K \otimes_C R \rightarrow M$ by the formula $\overline{\sigma}(a \otimes b) = a \sigma(b)$. Set $\psi=\overline{\sigma} \circ \phi\colon K \otimes_C E \rightarrow M$.

Observe that $\psi(K \otimes 1) \subset K$. Indeed for $a \in K$ we have
\begin{gather*}\psi(a \otimes 1)=(\overline{\sigma} \circ \phi)(a \otimes 1)=\overline{\sigma} (a \otimes 1) =a.\end{gather*}
Because $E$ consists of constants, then $\psi(1 \otimes E) \subset D_{\sigma}$ (regardless of the choice of the differential isomorphism $\sigma$).
So \begin{gather*}\psi(K \otimes_C E)=\psi \big( (K \otimes 1)(1 \otimes E)\big) \subset K D_{\sigma}.\end{gather*}

We have then the commutative diagram
\begin{gather*}
\xymatrix{K \otimes_C E \ar[rr]^{\phi} \ar[rrdd]^{\psi}&& K \otimes_C R \ar[dd]^{\overline{\sigma}}\\&& \\
&& K D_{\sigma}. }
\end{gather*}
Because $\phi$ is surjective, $\overline{\sigma}(K\otimes_C R) \subset KD_{\sigma}$, which implies that $\sigma K \subset K D_\sigma$. We can then use \cite[Proposition 12.5]{Kov2} and conclude that $K/C$
is a strongly normal extension.
\end{proof}

\begin{Remark}Let us observe that in order to prove that $K/C$ is a Picard--Vessiot extension it would be sufficient to know that $R$ is a differentially simple ring. In this case, the fact that $K/C$ is strongly normal and $K\otimes_C R$ is almost constant imply that $K/C$ is Picard--Vessiot.
\end{Remark}

\section{Application to polynomial automorphisms}\label{section3}

Let $C$ be a field of characteristic zero and let $F=(F_1, \ldots, F_n)\colon C^n \rightarrow C^n$ be a polynomial map such that $\mathrm{det}(J_F) \in C{\setminus} \{0\}$. We can equip
$C(X_1, \ldots, X_n)$ with the Nambu derivations, i.e., derivations $\delta_1, \ldots, \delta_n$ given by
\begin{gather*}
\left(
\begin{matrix}
\delta_1 \\
\vdots \\
\delta_n%
\end{matrix}
\right) = \big(J_F^{-1}\big)^T \left(
\begin{matrix}
\frac{\partial}{\partial X_1} \\
\vdots \\
\frac{\partial}{\partial X_n}
\end{matrix}
\right).
\end{gather*}
Observe that both the field
$C(F_1, \ldots, F_n)$ and the polynomial ring $C[X_1,\dots,X_n]$ are stable under $\delta_1, \ldots, \delta_n$.

\begin{Theorem}\label{th2}Let $C$ be an algebraically closed field of characteristic zero and let $F=(F_1, \ldots, \allowbreak F_n)\colon C^n \rightarrow C^n$ be a polynomial map such that $\det(J_F) \in C{\setminus} \{0\}$. Let $R$ $($respectively~$K)$ denote the polynomial ring $C[X]$ $($respectively the rational function field~$C(X))$ with the partial differential structure given by the Nambu derivations. We extend these derivations to the tensor product $K\otimes_C R$ and denote by $E$ the ring of constants of $K\otimes_C R$. If the differential ring homomorphism
\begin{gather*} \phi \colon \ K \otimes_C E \rightarrow K \otimes_C R, \qquad \phi (a \otimes d)=(a \otimes 1)d \end{gather*}
is an isomorphism, then $F$ is invertible and its inverse is a polynomial map.
\end{Theorem}

\begin{proof}By Theorem~\ref{mainthm}, the differential field extension $K/C$ is a strongly normal extension. If we consider the intermediate differential field $k=C(F_1,\ldots, F_n)$, then $K/k$ is again strongly normal. Now, since $\det(J_F) \in C{\setminus} \{0\}$ and the field $C$ has characteristic zero, the image of~$F$ is a Zariski open subset of the affine space $C^n$. Hence the fields $K=C(X)$ and $k=C(F)$ have the same transcendence degree over $C$. This implies that $K/k$ is an algebraic extension, and so it is a Galois extension. Then by Campbell's theorem~\cite{C}, $F$ is invertible and its inverse is a~polynomial map.
\end{proof}

\begin{Remark}By Proposition \ref{injective}, the map $\phi$ is injective. In order to prove that it is also surjective, it is enough to prove that the elements $1\otimes X_i$, $1\leq i \leq n$ lie in the image of~$\phi$. Hence Theorem~\ref{th2} provides an effective criterion to check the invertibility of polynomial maps. Finally when we know that~$F$ has a polynomial inverse, the ring $C[X]$ is the same as $C[F]$ and therefore it is the Picard--Vessiot ring over~$C$.
\end{Remark}

\begin{Remark} The criterion given in \cite[Proposition 3.1.4(i)]{E} establishes the equivalence of the invertibility of a polynomial map and the nilpotency of the derivation $D=Y_1\delta_1+\dots +Y_n \delta_n$, where $Y_1,\dots,Y_n$ are additional variables and $\delta_1,\dots,\delta_n$ are the Nambu derivations. We have compared this criterion to our criterion in~\cite{ABCH} by applying both to the polynomial map associated to $g_1$ in \cite[Example~5.6.8]{Bot}. We have observed that with our criterion the computation of the inverse took less than one second whereas with the other criterion the computation was not ended after one hour of running the program. To implement the criterion in \cite[Proposition 3.1.4]{E} both the computation of the Nambu derivations and the powers of the derivation $D$ implies a~big number of products and is therefore rather time consuming. This criterion is quite useful for more general rings of coefficients whereas our criterion works very well in positive characteristic. For more details on it see our recent paper~\cite{ABCH2}.
\end{Remark}

\subsection*{Acknowledgments}
This paper is dedicated to the memory of Jerald Joseph Kovacic. During his visit in Barcelona in summer 2008 he discussed with us algebraic aspects of the theory of strongly normal extensions. This work was partially supported by the Faculty of Applied Mathematics AGH UST statutory tasks within subsidy of Ministry of Science and Higher Education. Crespo and Hajto acknow\-ledge support by grant MTM2015-66716-P (MINECO/FEDER,UE). We thank the anonymous referees for their valuable remarks and suggestions.

\pdfbookmark[1]{References}{ref}
\LastPageEnding


\begin{thebibliography}{99}
\footnotesize\itemsep=0pt

\bibitem{ABCH}
Adamus E., Bogdan P., Crespo T., Hajto Z., An effective study of polynomial
 maps, \href{https://doi.org/10.1142/S0219498817501419}{\textit{J.~Algebra Appl.}} \textbf{16} (2017), 1750141, 13~pages,
 \href{https://arxiv.org/abs/1506.01654}{arXiv:1506.01654}.

\bibitem{ABCH2}
Adamus E., Bogdan P., Crespo T., Hajto Z., Pascal finite polynomial
 automorphisms, \href{https://doi.org/10.1142/S021949881950124X}{\textit{J.~Algebra Appl.}}, {t}o appear.

\bibitem{ABH}
Adamus E., Bogdan P., Hajto Z., An effective approach to {P}icard--{V}essiot
 theory and the {J}acobian conjecture, \href{https://doi.org/10.4467/20838476SI.17.004.8150}{\textit{Schedae Informaticae}}
 \textbf{26} (2017), 49--59, \href{https://arxiv.org/abs/1506.01662}{arXiv:1506.01662}.

\bibitem{Bass}
Bass H., Connell E.H., Wright D., The {J}acobian conjecture: reduction of
 degree and formal expansion of the inverse, \href{https://doi.org/10.1090/S0273-0979-1982-15032-7}{\textit{Bull. Amer. Math. Soc.
 (N.S.)}} \textbf{7} (1982), 287--330.

\bibitem{C}
Campbell L.A., A condition for a polynomial map to be invertible, \href{https://doi.org/10.1007/BF01349234}{\textit{Math.
 Ann.}} \textbf{205} (1973), 243--248.

\bibitem{CH}
Crespo T., Hajto Z., Picard--{V}essiot theory and the {J}acobian problem,
 \href{https://doi.org/10.1007/s11856-011-0145-y}{\textit{Israel~J. Math.}} \textbf{186} (2011), 401--406.

\bibitem{Bo}
de~Bondt M., Quasi-translations and counterexamples to the homogeneous
 dependence problem, \href{https://doi.org/10.1090/S0002-9939-06-08335-3}{\textit{Proc. Amer. Math. Soc.}} \textbf{134} (2006),
 2849--2856.

\bibitem{Bot}
de~Bondt M., Homogeneous {K}eller maps, Ph.D.~Thesis, {R}adboud University
 Nijmegen, 2007, available at
 \url{http://webdoc.ubn.ru.nl/mono/b/bondt_m_de/homokema.pdf}.

\bibitem{Ke}
Keller O.-H., Ganze {C}remona-{T}ransformationen, \href{https://doi.org/10.1007/BF01695502}{\textit{Monatsh. Math. Phys.}}
 \textbf{47} (1939), 299--306.

\bibitem{Kov2}
Kovacic J.J., The differential {G}alois theory of strongly normal extensions,
 \href{https://doi.org/10.1090/S0002-9947-03-03306-3}{\textit{Trans. Amer. Math. Soc.}} \textbf{355} (2003), 4475--4522.

\bibitem{Kov1}
Kovacic J.J., Geometric characterization of strongly normal extensions,
 \href{https://doi.org/10.1090/S0002-9947-06-03868-2}{\textit{Trans. Amer. Math. Soc.}} \textbf{358} (2006), 4135--4157.

\bibitem{L}
Levelt A.H.M., Differential {G}alois theory and tensor products, \href{https://doi.org/10.1016/0019-3577(90)90012-C}{\textit{Indag.
 Math. (N.S.)}} \textbf{1} (1990), 439--449.

\bibitem{Sma}
Smale S., Mathematical problems for the next century, \href{https://doi.org/10.1007/BF03025291}{\textit{Math.
 Intelligencer}} \textbf{20} (1998), 7--15.

\bibitem{E2}
van~den Essen A., Polynomial automorphisms and the {J}acobian conjecture, in
 Alg\`ebre Non Commutative, Groupes Quantiques et Invariants ({R}eims, 1995),
 \textit{S\'{e}min. Congr.}, Vol.~2, Soc. Math. France, Paris, 1997, 55--81.

\bibitem{E}
van~den Essen A., Polynomial automorphisms and the {J}acobian conjecture,
 \textit{Progress in Mathematics}, Vol.~190, \href{https://doi.org/10.1007/978-3-0348-8440-2}{Birkh\"{a}user Verlag}, Basel,
 2000.

\bibitem{W}
Wang S.S.-S., A {J}acobian criterion for separability, \href{https://doi.org/10.1016/0021-8693(80)90233-1}{\textit{J.~Algebra}}
 \textbf{65} (1980), 453--494.

\bibitem{Y}
Yagzhev A.V., Keller's problem, \href{https://doi.org/10.1007/BF00973892}{\textit{Siberian Math.~J.}} \textbf{21} (1980),
 747--754.

\end{thebibliography}
\end{document}